\documentclass[a4paper, 12pt]{amsart}
\usepackage{mathrsfs}
\usepackage{amsmath}
\usepackage{amssymb}
\usepackage{verbatim}
\usepackage{xcolor}
\usepackage{CJK}
\usepackage{graphicx}
\usepackage[all]{xy}
\DeclareGraphicsExtensions{.eps,.ps,.jpg,.bmp}

\newtheorem{theorem}{Theorem}[section]
\newtheorem{proposition}{Proposition}[section]
\newtheorem{corollary}{Corollary}[section]

\newtheorem{remark}{Remark}[section]

\theoremstyle{definition}
\newtheorem{definition}{Definition}[section]

\theoremstyle{remark}
\newtheorem*{acknowledgment}{Acknowledgment}

\title{A Koll\'{a}r-type vanishing theorem}
\author{Jingcao Wu}

\begin{document}

\begin{abstract}
Let $f:X\rightarrow Y$ be a smooth fibration between two complex manifolds $X$ and $Y$, and let $L$ be a pseudo-effective line bundle on $X$. We obtain a sufficient condition for $R^{q}f_{\ast}(K_{X/Y}\otimes L)$ to be reflexive and hence derive a Koll\'{a}r-type vanishing theorem.
\end{abstract}
\maketitle
\pagestyle{plain}

\section{Introduction}
Let $f:X\rightarrow Y$ be a fibration between two complex manifolds $X$ and $Y$. Koll\'{a}r proved a vanishing theorem of the associated higher direct images $R^{q}f_{\ast}(K_{X/Y})$ in his remarkable paper \cite{Ko86a}.
\begin{theorem}(Koll\'{a}r, \cite{Ko86a})\label{t1}
Let $f:X\rightarrow Y$ be a surjective map between a projective manifold $X$ and a projective variety $Y$. If $A$ is an ample divisor on $Y$, then for any $i>0$ and $q\geqslant0$,
\begin{equation*}
 H^{i}(Y,R^{q}f_{\ast}(K_{X})\otimes\mathcal{O}_{Y}(A))=0.
\end{equation*}
\end{theorem}

It is meaningful to consider the similar properties of the adjoint bundle $K_{X}\otimes L$ for a line bundle $L$ on $X$. When $L$ is endowed with a smooth (semi-)positive Hermitian metric, it was shown in \cite{MT08,Tk95} that the Koll\'{a}r-type vanishing theorem also holds. So we are interested in the singular case in this paper. In fact, such a vanishing theorem was first developed by Kawamata as follows.
\begin{theorem}(Kawamata, c.f. Theorem 2.86 in \cite{FEM13})\label{t2}
Let $f:X\rightarrow Y$ be a surjective map between a projective manifold $X$ and a projective variety $Y$. Assume that $L$ is a line bundle on $X$ which is numerically equivalent to a $\mathbb{Q}$-divisor $D$ with simple normal crossings and satisfying $\lfloor D\rfloor=0$. If $A$ is an ample divisor on $Y$, then for any $i>0$ and $q\geqslant0$,
\begin{equation*}
H^{i}(Y,R^{q}f_{\ast}(K_{X}\otimes L)\otimes\mathcal{O}_{Y}(A))=0.
\end{equation*}
\end{theorem}

Similar results also appeared in \cite{Mat16,Ohs84}. In the proof of these Koll\'{a}r-type vanishing theorems, there are two key ingredients: the injectivity theorem and the torsion-freeness of the higher direct images. The injectivity theorem is due to Koll\'{a}r \cite{Ko86a}. It was generalized by Esnault and Viehweg \cite{EV92}, and was further strengthened by Ambro \cite{Am14}. Furthermore, there is an analytic version of the injectivity theorem in \cite{GoM17}, which is the start point of this paper.
\begin{theorem}(Gongyo--Matsumura, \cite{GoM17,Mat18})\label{t3}
Let $F$ and $L$ be two line bundles on a compact K\"{a}hler manifold $X$ with (singular) Hermitian metrics $h_{F}$ and $h_{L}$, respectively, of semi-positive curvature. Assume that there exists an $\mathbb{R}$-effective divisor $\Delta$ on $X$ such that $h_{F}=h^{a}_{L}h_{\Delta}$ for a positive real number $a$ and the singular metric $h_{\Delta}$ defined by $\Delta$. Then for a section $s$ of $L$ satisfying $\sup_{X}|s|_{h_{L}}<\infty$, the multiplication map induced by $s$
\begin{equation*}
   \Phi_{s}:H^{q}(X,K_{X}\otimes F\otimes\mathscr{I}(h_{F}))\stackrel{\otimes s}\rightarrow H^{q}(X,K_{X}\otimes F\otimes L\otimes\mathscr{I}(h_{F}h_{L}))
\end{equation*}
is injective for an integer $q$ with $0\leqslant q\leqslant \dim X$.
\end{theorem}
Here and in the rest of this paper, we use $\mathscr{I}(h)$ to denote the multiplier ideal sheaf \cite{Dem12} associated with a (singular) metric $h$.

One may ask whether $R^{q}f_{\ast}(K_{X/Y}\otimes L)$ is torsion-free when $L$ is pseudo-effective. In general, we cannot get the desired result because the singular part of $L$ after being pushed forward to $Y$ cannot be controlled. However, when the singularity of $L$ is sufficiently mild, we can establish the following theorem.
\begin{theorem}\label{t4}
Let $(L,h)$ be a pseudo-effective line bundle on a complex manifold $X$ so that there exists a section $s$ of some positive multiple $mL$ satisfying $\sup_{X}|s|_{h^{m}}<\infty$. Assume that $Y$ is a complex manifold, and $f: X\rightarrow Y$ is a smooth, K\"{a}hler fibration with connected compact fibres such that for any $y\in Y$, $h_{y}:=h|_{X_{y}}$ is well-defined and
\begin{equation*}
    H^{q}(X_{y},K_{X_{y}}\otimes L_{y})=H^{q}(X_{y},(K_{X_{y}}\otimes L_{y})\otimes\mathscr{I}(h_{y})).
\end{equation*}
Then $R^{q}f_{\ast}(K_{X/Y}\otimes L)$ is reflexive.
\end{theorem}

Combining the two theorems above, we then can prove the following Koll\'{a}r-type vanishing theorem.
\begin{theorem}\label{t5}
Let $f:X\rightarrow Y$ be a surjective fibration between two projective manifolds $X$ and $Y$. Let $(L,h)$ be a $\mathbb{Q}$-effective line bundle on $X$ so that $\mathscr{I}(h)=\mathcal{O}_{X}$. If $A$ is an ample divisor on $Y$, then for any $i>0$ and $q\geqslant0$,
\begin{equation*}
    H^{i}(Y,R^{q}f_{\ast}(K_{X}\otimes L)\otimes\mathcal{O}_{Y}(A))=0.
\end{equation*}
\end{theorem}

Obviously, Theorem \ref{t2} can be considered as a special case of Theorem \ref{t5}. However, it is difficult to derive Theorem \ref{t5} directly from Theorem \ref{t2}. Indeed, for a $\mathbb{Q}$-effective line bundle $L$ with $\mathscr{I}(L)=\mathcal{O}_{X}$, there exists a log-resolution $g:\Tilde{X}\rightarrow X$ such that the pullback of $L$ by $g$ is a $\mathbb{Q}$-divisor $D$ on $\Tilde{X}$ with simple normal crossings and satisfying $\lfloor D\rfloor=0$. Therefore if one would like to apply Theorem \ref{t2} to get Theorem \ref{t5}, it is inevitable to prove the degeneration of the Leray spectral sequence
\begin{equation*}
    R^{p}f_{\ast}R^{q}g_{\ast}(K_{Z}\otimes D)\Rightarrow R^{p+q}(f\circ g)_{\ast}(K_{Z}\otimes D),
\end{equation*}
which is usually non-trivial. One can see some of the difficulties in \cite{Tk95} where the smooth case is proved. Hence we use our method stated above to prove Theorem \ref{t5}.

We also remark that Theorem \ref{t5} can be used to prove the positivity of $R^{q}f_{\ast}(K_{X}\otimes L)$. Indeed, the canonical vanishing theorem says that for a nef vector bundle $E$ and an ample line bundle $A$ on $Y$,
\begin{equation*}
  H^{i}(Y,K_{Y}\otimes E\otimes A)=0.
\end{equation*}
Hence Theorem \ref{t5} implies that the higher direct images have the nef property in the sense of cohomology. In fact, we have the following result.
\begin{theorem}\label{t6}
Under the same assumptions as in Theorem \ref{t5}, if $A$ is an ample and globally generated line bundle and $A^{\prime}$ is a nef line bundle on $Y$, then the sheaf $R^{q}f_{\ast}(K_{X}\otimes L)\otimes A^{m}\otimes A^{\prime}$ is globally generated for any $q\geqslant0$ and $m\geqslant\dim Y+1$.
\end{theorem}

The paper is arranged as follows: In section 2 we first prove an embedding theorem, and then prove Theorem \ref{t4}. In section 3 we prove Theorems \ref{t5} and \ref{t6}.
\begin{acknowledgment}
The author sincerely thanks his supervisor Professor Jixiang Fu for discussions. Thanks also go to Shin-ichi Matsumura, who kindly provided some comment about the reference of this paper. Finally, he is very grateful to the referee for many useful suggestions on how to improve the paper.
\end{acknowledgment}

\section{Embedding Theorem}
Let $f:X\rightarrow Y$ be a smooth, K\"{a}hler fibration between two complex manifolds $X$ and $Y$, and let $L$ be a pseudo-effective line bundle on $X$. In order to prove that $R^{q}f_{\ast}(K_{X/Y}\otimes L)$ is reflexive, we will embed it into another reflexive sheaf so that the embedding is split. We first consider how to embed the cohomology group $H^{q}(X_{y},K_{X_{y}}\otimes L_{y})$ for any $y\in Y$ by the definition of the higher direct images.

If $L$ is semi-positive, it was done in \cite{MT08,Tk95}. In this case, since $f$ is K\"{a}hler, it gives a $(1,1)$-form $\omega_{f}$ on $X$ such that $\omega_{y}:=\omega_{f}|_{X_{y}}$ is a K\"{a}hler form on each fibre $X_{y}$. Hence, for a cohomology class $[u]\in H^{q}(X_{y},K_{X_{y}}\otimes L_{y})$, one can take the harmonic representative $\tilde{u}$ of the class $[u]$, and get an element $\ast\tilde{u}\in H^{0}(X_{y},\Omega^{n-q}_{X_{y}}\otimes L_{y})$. Here $\ast$ is the Hodge star operator defined by $\omega_{y}$. In other words, we have a map
\begin{equation*}
   S^{q}_{y}: R^{q}f_{\ast}(K_{X/Y}\otimes L)_{y}\rightarrow f_{\ast}(\Omega^{n-q}_{X/Y}\otimes L)_{y}
\end{equation*}
defined by $S^{q}_{y}([u])=\ast\tilde{u}$. On the other hand, there is a natural map
\begin{equation*}
   L^{q}_{y}:f_{\ast}(\Omega^{n-q}_{X/Y}\otimes L)_{y}\rightarrow R^{q}f_{\ast}(K_{X/Y}\otimes L)_{y}
\end{equation*}
defined by $L^{q}_{y}(v)=[v\wedge\omega^{q}_{y}]$. One can easily check that $L^{q}_{y}\circ S^{q}_{y}=id$. So the map $S^{q}_{y}$ is split.

Now assume that $L$ is merely semi-positive in the sense of current. So the classic harmonic theory cannot be used. In order to handle this case, we use Demailly's analytic approximation in \cite{DPS01}. Briefly, we approximate the original singular metric by a family of singular metrics which are smooth on a Zariski open set $W$. We use these metrics to define a family of maps for $y\in W$ like the $S^{q}_{y}$ above, and then take the limit to get the desired map. The existence of the limit is guaranteed by the assumption on the singularity of the line bundle $L$. More precisely, we have the following result.

\begin{proposition}\label{p1}
Let $(L,h)$ be a pseudo-effective line bundle on a compact K\"{a}hler manifold $X$. Assume that there exists a section $s$ of some positive multiple $mL$ satisfying $\sup_{X}|s|_{h^{m}}<\infty$. Then for any integer $q$ with $0\leqslant q\leqslant n$, there exists an injective morphism
\begin{equation*}
   S^{q}:H^{q}(X,K_{X}\otimes L\otimes\mathscr{I}(h))\rightarrow H^{0}(X,\Omega^{n-q}_{X}\otimes L).
\end{equation*}
\begin{proof}
Let $\omega$ be a K\"{a}hler form on $X$. By Theorem 2.2.1 in \cite{DPS01} we can approximate $h$ by a family of singular metrics $\{h_{\varepsilon}\}_{\varepsilon>0}$ with the following properties:

(a) $h_{\varepsilon}$ is smooth on $X-Z_{\varepsilon}$ for a subvariety $Z_{\varepsilon}$;

(b) $h_{\varepsilon_{2}}\leqslant h_{\varepsilon_{1}}\leqslant h$ holds for any $0<\varepsilon_{1}\leqslant\varepsilon_{2}$;

(c) $\mathscr{I}(h)=\mathscr{I}(h_{\varepsilon})$; and

(d) $i\Theta_{h_{\varepsilon}}(L)\geqslant-\varepsilon\omega$.

Thanks to the proof of the openness conjecture by Berndtsson \cite{Ber15}, one can arrange $h_{\varepsilon}$ with logarithmic poles along $Z_{\varepsilon}$ according to the remark in \cite{DPS01}. Moreover, since the norm $|s|_{h^{m}}$ is bounded on $X$, the set $\{x\in X|\nu(h_{\varepsilon},x)>0\}$ for every $\varepsilon>0$ is contained in the subvariety $Z:=\{x|s(x)=0\}$ by property (b). Here $\nu(h_{\varepsilon},x)$ refers to the Lelong number of $h_{\varepsilon}$ at $x$. Hence, instead of (a), we can assume that

(a') $h_{\varepsilon}$ is smooth on $X-Z$ and has logarithmic poles along $Z$, where $Z$ is a subvariety of $X$ independent of $\varepsilon$.

Now let $W=X-Z$. We can use the method in \cite{Dem82} to construct a complete K\"{a}hler metric on $W$ as follows. Take a quasi-psh function $\psi(\leqslant-e)$ on $X$ such that it is smooth on $W$ and has logarithmic pole along $Z$. Then $\Psi:=\log^{-1}(-\psi)$ is bounded on $X$. Define $\tilde{\omega}=\omega+\frac{1}{l}i\partial\bar{\partial}\Psi$ for some $l\gg0$. It is easy to verify that $\tilde{\omega}$ is a complete K\"{a}hler metric on $W$ and $\tilde{\omega}\geqslant\frac{1}{l}\omega$.

Let $L^{n,q}_{(2)}(W,L)_{h_{\varepsilon,\tilde{\omega}}}$ be the $L^{2}$-space of the $L$-valued $(n,q)$-forms $u$ on $W$ with respect to the inner product $\|\cdot\|_{h_{\varepsilon,\tilde{\omega}}}$ defined by
$$\|u\|^{2}_{h_{\varepsilon},\tilde{\omega}}=\int_{W}|u|^{2}_{h_{\varepsilon},\tilde{\omega}}dV_{\tilde{\omega}}.$$
Then we have the orthogonal decomposition
\begin{equation}\label{e1}
L^{n,q}_{(2)}(W,L)_{h_{\varepsilon},\tilde{\omega}}=\mathrm{Im}\bar{\partial}\bigoplus\mathcal{H}^{n,q}_{h_{\varepsilon}, \tilde{\omega}}(L)\bigoplus\mathrm{Im}\bar{\partial}^{\ast_{\tilde{\omega}}}_{h_{\varepsilon}}
\end{equation}
where
\begin{equation*}
    \mathcal{H}^{n,q}_{h_{\varepsilon}, \tilde{\omega}}(L)=\{u|\bar{\partial}u=0, \bar{\partial}^{\ast_{\tilde{\omega}}}_{h_{\varepsilon}}u=0\}.
\end{equation*}
We give some explanation for decomposition (\ref{e1}). Usually $\mathrm{Im}\bar{\partial}$ is not closed in the $L^{2}$-space of a noncompact manifold even if the metric is complete. However, in the situation we consider here, $W$ has the compactification $X$ and the forms are bounded in $L^{2}$-norm. In fact, by Claim 1 in \cite{Fuj12}, we have the isomorphism
\begin{equation}\label{e2}
H^{q}(X,K_{X}\otimes L\otimes\mathscr{I}(h_{\varepsilon}))\cong\frac{L^{n,q}_{(2)}(W,L)_{h_{\varepsilon,\widetilde{\omega}}}\cap\mathrm{Ker}\bar{\partial}}{\mathrm{Im}
\bar{\partial}}
\end{equation}
from which we can see that $\mathrm{Im}\bar{\partial}$ as well as $\mathrm{Im}\bar{\partial}^{\ast_{\tilde{\omega}}}_{h_{\varepsilon}}$ is closed. Hence decomposition (\ref{e1}) holds.

Isomorphism (\ref{e2}) was constructed in \cite{Fuj12} by the standard diagram chasing and $L^{2}$-extension technique. Such ideas more or less appeared in \cite{Siu84} and other related papers. Here we give the sketch of its proof.

Take a finite Stein cover $\mathcal{U}=\{U_{i}\}$ of $X$. By Cartan and Leray, we have the isomorphism
\begin{equation*}
   H^{q}(X,K_{X}\otimes L\otimes\mathscr{I}(h_{\varepsilon}))\cong\check{H}^{q}(\mathcal{U},K_{X}\otimes L\otimes\mathscr{I}(h_{\varepsilon})),
\end{equation*}
where the right hand side is the \v{C}ech cohomology group calculated by $\mathcal{U}$. By the standard diagram chasing, we have a homomorphism
\begin{equation*}
   \Bar{\alpha}:\check{H}^{q}(\mathcal{U},K_{X}\otimes L\otimes\mathscr{I}(h_{\varepsilon}))\rightarrow\frac{L^{n,q}_{(2)}(W,L)_{h_{\varepsilon,\tilde{\omega}}}\cap\mathrm{Ker}\bar{\partial}}{\mathrm{Im}\bar{\partial}}.
\end{equation*}
On the other hand, for $w\in L^{n,q}_{(2)}(W,L)_{h_{\varepsilon,\tilde{\omega}}}\cap\mathrm{Ker}\overline{\partial}$, we denote $w^{0}_{i_{0}}:=w|_{U_{i_{0}}\cap W}$ and solve the $\bar{\partial}$-equation $\bar{\partial}w^{1}_{i_{0}}=w^{0}_{i_{0}}$ on $U_{i_{0}}\cap W$ with certain $L^{2}$-estimate. Denote $w^{1}=\{w^{1}_{i_{0}}\}$. Since $\bar{\partial}(\delta w^{1})=0$, we obtain $w^{2}$ such that $\bar{\partial}w^{2}=\delta w^{1}$ on each $U_{i_{0}i_{1}}\cap W$. Here for convenience, we have put $U_{i_{0}...i_{q}}=U_{i_{0}}\cap\cdot\cdot\cdot\cap U_{i_{q}}$. By repeating this procedure, we obtain $w^{q}$ so that $\bar{\partial}w^{q}=\delta w^{q-1}$. Put $v:=\delta w^{q}=\{v_{i_{0}...i_{q}}\}$. Notice that $v_{i_{0}...i_{q}}$ is an $L$-valued $(n,0)$-form on $U_{i_{0}...i_{q}}\cap W$ with bounded $L^{2}$-norm and $\delta v=0$. Therefore we have a homomorphism
\begin{equation*}
   \Bar{\beta}:\frac{L^{n,q}_{(2)}(W,L)_{h_{\varepsilon,\tilde{\omega}}}\cap\mathrm{Ker}\bar{\partial}}{\mathrm{Im}
   \bar{\partial}}\rightarrow\check{H}^{q}(\mathcal{U},K_{X}\otimes L\otimes\mathscr{I}(h_{\varepsilon})).
\end{equation*}
It is not difficult to check that $\Bar{\alpha}$ and $\Bar{\beta}$ induce the desired isomorphism.

Now we define the map $S^{q}$. We use the de Rham--Weil isomorphism
\begin{equation*}
    H^{q}(X,K_{X}\otimes L\otimes\mathscr{I}(h))\cong\frac{\mathrm{Ker}\bar{\partial}\cap L^{n,q}_{(2)}(X,L)_{h,\omega}}{\mathrm{Im}\bar{\partial}}
\end{equation*}
to represent a given cohomology class by a $\bar{\partial}$-closed $L$-valued $(n,q)$-form $u$ with $\|u\|_{h,\omega}<\infty$. We denote $u|_{W}$ simply by $u_{W}$. Since $\tilde{\omega}\geqslant\frac{1}{l}\omega$, it is easy to verify that $|u_{W}|^{2}_{h_{\varepsilon},\tilde{\omega}}dV_{\tilde{\omega}}\leqslant C|u|^{2}_{h_{\varepsilon},\omega}dV_{\omega}$, which leads to the inequality $\|u_{W}\|_{h_{\varepsilon},\tilde{\omega}}\leqslant C\|u\|_{h_{\varepsilon,\omega}}$. Here $C$ is a constant used in a generic sense. Hence by property (b), we have
$\|u_{W}\|_{h_{\varepsilon},\tilde{\omega}}\leqslant C\|u\|_{h,\omega}$ which implies $u_{W}\in L^{n,q}_{(2)}(W,L)_{h_{\varepsilon},\tilde{\omega}}$. By decomposition (\ref{e1}), we have the harmonic representative $u_{\varepsilon}$ of $[u_{W}]$ in $\mathcal{H}^{n,q}_{h_{\varepsilon,\tilde{\omega}}}(L)$ and hence
\begin{equation*}
    \ast_{\tilde{\omega}}u_{\varepsilon}\in H^{0}(W,\Omega^{n-q}_{W}\otimes L\otimes\mathscr{I}(h_{\varepsilon})).
\end{equation*}
Since $\|\!\ast_{\tilde{\omega}}\! u_{\varepsilon}\|_{h_{\varepsilon,\tilde{\omega}}}=\|u_{\varepsilon}\|_{h_{\varepsilon},\tilde{\omega}}\leqslant C\|u\|_{h,\omega}$, there exists a subsequence of $u_{\varepsilon}$, which is still denoted by $u_{\varepsilon}$, and a current $v$ such that
\begin{equation*}
    \lim_{\varepsilon\rightarrow0}\ast_{\tilde{\omega}}u_{\varepsilon}=v\in L^{n-q,0}_{(2)}(W,L)_{h,\tilde{\omega}}
\end{equation*}
in the sense of the weak $L^{2}$-topology. Moreover, for any test form $w$ on $W$, we have
\begin{equation*}
\begin{split}
    (\bar\partial v,w)=(v,\bar\partial^{\ast}w)=\lim_{\varepsilon\to 0}(\ast u_{\varepsilon},\bar\partial^{\ast}w)=\lim_{\varepsilon\to 0}(\bar\partial(\ast u_\varepsilon), w)= 0
\end{split}
\end{equation*}
by the definition of the weak convergence. Here $\bar\partial^{\ast}$ is the formal adjoint operator. Hence $v$ is actually a holomorphic form on $W$ with $\|v\|_{h_{\varepsilon},\tilde{\omega}}\leqslant C\|u\|_{h,\omega}$. By the well-known extension theorem (such as Proposition 1.14 in \cite{Ohs02}), we can extend $v$ to the whole space $X$ to get an element, which is still denoted by $v$, in $H^{0}(X,\Omega^{n-q}_{X}\otimes L)$. (The weight functions here are singular, but since they are bounded above the extension result is still true.) We define $S^{q}([u])=v$.

In the following we prove that the map $S^{q}$ is injective. Assume $S^{q}([u])$ is zero. We need to prove the representative
\begin{equation*}
   u\in\mathrm{Ker}\bar{\partial}\cap L^{n,q}_{(2)}(X,L)_{h,\omega}
\end{equation*}
is also in $\mathrm{Im}\bar{\partial}$. Since $\|u_{W}\|_{h_{\varepsilon},\tilde{\omega}}\leqslant C\|u\|_{h,\omega}$, it follows that $u_{W}\in L^{n,q}_{(2)}(W,L)_{h_{\varepsilon},\tilde{\omega}}$ for every $\varepsilon>0$. Hence there exist $u_{\varepsilon}\in\mathcal{H}^{n,q}_{h_{\varepsilon},\tilde{\omega}}(L)$ and $w_{\varepsilon}\in\mathrm{Dom}\bar{\partial}\cap L^{n,q-1}_{(2)}(W,L)_{h_{\varepsilon},\tilde{\omega}}$ such that $u_{W}=\bar{\partial}w_{\varepsilon}+u_{\varepsilon}$.

Fix $\varepsilon_{0}>0$. For any $\varepsilon$ with $0<\varepsilon<\varepsilon_{0}$, by property (b) we have
\begin{equation*}
    \|u_{\varepsilon}\|_{h_{\varepsilon_{0}},\tilde{\omega}}\leqslant\|u_{\varepsilon}\|_{h_{\varepsilon},\tilde {\omega}}\leqslant C\|u\|_{h,\omega}.
\end{equation*}
Hence the $L^{2}$-norm of $u_{\varepsilon}$ with respect to the metrics $h_{\varepsilon_{0}}$ and $\tilde{\omega}$ is uniformly bounded. So there exists a subsequence of $u_{\varepsilon}$ and a $u_{0}\in L^{n,q}_{(2)}(W,L)_{h_{\varepsilon_{0}},\tilde{\omega}}$ such that $\lim_{\varepsilon\rightarrow0} u_{\varepsilon}=u_{0}$ in the sense of the weak $L^{2}$-topology. We claim that $u_{0}$ is zero. In fact, by the definition of $S^{q}$, we have $\lim_{\varepsilon\rightarrow0}\ast_{\tilde{\omega}}u_{\varepsilon}=S^{q}(u)=0$. Hence from the identity $\|\!\ast_{\tilde{\omega}}\! u_{\varepsilon}\|_{h_{\varepsilon_{0}},\tilde{\omega}}=\|u_{\varepsilon}\|_{h_{\varepsilon_{0}},\tilde{\omega}}$ we have $\lim_{\varepsilon\rightarrow0}\|u_{\varepsilon}\|_{h_{\varepsilon_{0}},\widetilde{\omega}}=0$ which implies $u_{0}=0$. We then claim that
\begin{equation}\label{e7}
    u_{W}\in \mathrm{Im}\bar{\partial}\cap L^{n,q}_{(2)}(W,L)_{h_{\varepsilon_{0}},\tilde{\omega}}.
\end{equation}
Indeed, since
\begin{equation*}
   u_{W}=\lim_{\varepsilon\rightarrow0}\bar{\partial}w_{\varepsilon}+\lim_{\varepsilon\rightarrow0} u_{\varepsilon}=\lim_{\varepsilon\rightarrow0}\bar{\partial}w_{\varepsilon},
\end{equation*}
for any $w=w_{1}+\bar{\partial}^{\ast_{\tilde{\omega}}}_{h_{\varepsilon_{0}}}w_{2}\in\mathcal{H}^{n,q}_{h_{\varepsilon_{0}},\tilde {\omega}}(L)\bigoplus \mathrm{Im}\bar{\partial}^{\ast_{\tilde{\omega}}}_{h_{\varepsilon_{0}}}$ we have
\begin{equation*}
    (u_{W},w)=\lim_{\varepsilon\rightarrow0}(\bar{\partial}w_{\varepsilon},w_{1}+\bar{\partial}^{\ast_{\tilde{\omega}}}_{h_{\varepsilon_{0}}}w_{2})=0.
\end{equation*}
Hence $u_{W}$ is orthogonal to the space $\mathcal{H}^{n,q}_{h_{\varepsilon_{0}},\tilde {\omega}}(L)\bigoplus \mathrm{Im}\bar{\partial}^{\ast_{\tilde{\omega}}}_{h_{\varepsilon_{0}}}$.

We now use (\ref{e7}) to prove $u\in \mathrm{Im}\bar{\partial}\cap L^{n,q}_{(2)}(X,L)_{h,\omega}$. In fact, we have the following commutative diagram:
\begin{equation*}
\xymatrix{
    \frac{\mathrm{Ker}\bar{\partial}\cap L^{n,q}_{(2)}(X,L)_{h,\omega}}{\mathrm{Im}\bar{\partial}}\ar[d]_{f_{1}} & \stackrel{j}{\longrightarrow} & \frac{\mathrm{Ker}\bar{\partial}\cap L^{n,q}_{(2)}(W,L)_{h_{\varepsilon_{0},\tilde{\omega}}}}{\mathrm{Im}\bar{\partial}} \ar[d]^{f_{2}} \\
    \check{H}^{q}(\mathcal{U},K_{X}\otimes L\otimes\mathscr{I}(h)) & =\!=\!= &\check{H}^{q}(\mathcal{U},K_{X}\otimes L\otimes\mathscr{I}(h_{\varepsilon_{0}}))
    }
\end{equation*}
Here $j$ is induced by the restriction from $L^{n,\cdot}_{(2)}(X,L)_{h,\omega}$ to $L^{n,\cdot}_{(2)}(W,L)_{h_{\varepsilon_{0}},\tilde{\omega}}$, and $f_{i},i=1,2,$ is the de Rham--Weil isomorphisim to the \v{C}ech cohomology group. (Here $f_{2}$ is just the $\bar{\beta}$ constructed before, and one can also consult \cite{Fuj12} for more details). The bottom equality is obtained from the property (c) that $\mathscr{I}(h_{\varepsilon_{0}})=\mathscr{I}(h)$. It follows that $u$ goes to zero through the map $j$. Therefore $u$ goes to zero through $f_{1}$.
\end{proof}
\end{proposition}

Now the embedding theorem we need is a direct consequence of Proposition \ref{p1}.
\begin{corollary}\label{c1}
Let $(L,h)$ be a pseudo-effective line bundle on a complex manifold $X$. Assume that there exists a section $s$ of some positive multiple $mL$ satisfying $\sup_{X}|s|_{h^{m}}<\infty$. Let $Y$ be a complex manifold, and $f: X\rightarrow Y$ a smooth, K\"{a}hler fibration with connected compact fibres. If for any $y\in Y$, $h_{y}:=h|_{X_{y}}$ is well-defined and
\begin{equation*}
    H^{q}(X_{y},K_{X_{y}}\otimes L_{y})=H^{q}(X_{y},(K_{X_{y}}\otimes L_{y})\otimes\mathscr{I}(h_{y})),
\end{equation*}
then there exists a natural injective morphism
\begin{equation*}
   S^{q}:R^{q}f_{\ast}(K_{X/Y}\otimes L)\rightarrow f_{\ast}(\Omega^{n-q}_{X/Y}\otimes L).
\end{equation*}
\begin{proof}
For any $y\in Y$, replacing $X$ in Proposition \ref{p1} by $X_{y}$, we have the morphism
\begin{equation}\label{e8}
    S^{q}_{y}:H^{q}(X_{y},K_{X_{y}}\otimes L_{y})\rightarrow H^{0}(X_{y},\Omega^{n-q}_{X_{y}}\otimes L_{y}).
\end{equation}
It naturally induces an injective morphism from $R^{q}f_{\ast}(K_{X/Y}\otimes L)_{y}$ to $f_{\ast}(\Omega^{n-q}_{X/Y}\otimes L)_{y}$, which is still denoted by $S^{q}_{y}$. Then we get the desired morphism $S^{q}$ as $y$ varies.
\end{proof}
\end{corollary}

Moreover, it can be proved that this injective morphism is actually split.
\begin{proposition}\label{p3}
With the same assumptions as in Corollary \ref{c1}, there is a surjective morphism
\begin{equation*}
    L^{q}:f_{\ast}(\Omega^{n-q}_{X/Y}\otimes L)\rightarrow R^{q}f_{\ast}(K_{X/Y}\otimes L)
\end{equation*}
such that $L^{q}\circ S^{q}=id$, hence $S^{q}$ is split.
\begin{proof}
We first define at each point $y\in Y$ the morphism
\begin{equation*}
 \begin{split}
     L^{q}_{y}:H^{0}(X_{y},\Omega^{n-q}_{X_{y}}\otimes L_{y})&\rightarrow H^{q}(X_{y},K_{X_{y}}\otimes L_{y})
 \end{split}
\end{equation*}
by $L^{q}_{y}(v)=[\omega^{q}_{y}\wedge v]$. Remember $\omega_{f}$ is the $(1,1)$-form on $X$ given by the K\"{a}hler fibration $f$ and $\omega_{y}=\omega_{f}|_{X_{y}}$. We need to prove that for any $[u]\in H^{q}(X_{y},(K_{X_{y}}\otimes L_{y}))$, $L^{q}_{y}\circ S^{q}_{y}([u])=[u]$ with the morphism $S^{q}_{y}$ in (\ref{e8}).

When replacing $X$ in Proposition \ref{p1} by $X_{y}$, we have the notations such as $W_{y}$ and $\tilde{\omega}_{y}$ which correspond to $W$ and $\tilde{\omega}$ in Proposition \ref{p1} respectively. Denote $S^{q}_{y}([u])$ by $v$ and take $\tilde{\omega}^{q}_{y}\wedge v\in[(\omega^{q}_{y}\wedge v)|_{W_{y}}]$. Then
\begin{equation*}
    \tilde{\omega}^{q}_{y}\wedge v=\tilde{\omega}^{q}_{y}\wedge\lim_{\varepsilon\rightarrow0}\ast u_{\varepsilon}=\lim_{\varepsilon\rightarrow0}(\tilde{\omega}^{q}_{y}\wedge\ast u_{\varepsilon})=\lim_{\varepsilon\rightarrow0} u_{\varepsilon}.
\end{equation*}
Hence it is clear that $S^{q}_{y}\circ L^{q}_{y}(v)=v$ by tracing the definition of $S^{q}_{y}$. So $L^{q}_{y}(v)=[u]$ by the injectivity of $S^{q}_{y}$. Now all of the morphisms $L^{q}_{y}$ as $y$ varies naturally induce the desired morphism
\begin{equation*}
   L^{q}:f_{\ast}(\Omega^{n-q}_{X/Y}\otimes L)\rightarrow R^{q}f_{\ast}(K_{X/Y}\otimes L).
\end{equation*}
\end{proof}
\end{proposition}

\begin{remark}
This proposition also tells us that the definition of $S^{q}_{y}$ is actually canonical, i.e. it does not depend on the choice of $\{h_{\varepsilon}\}$ and $\tilde{\omega}$.
\end{remark}

By the preparation above, we can prove Theorem \ref{t4} by using a result in \cite{Har80}.
\begin{proof}[Proof of Theorem \ref{t4}]
Set $\mathcal{Q}^{q}_{y}:=\mathrm{Ker} L^{q}_{y}$. Since $S^{q}_{y}$ as well as $L^{q}_{y}$ is split, we have the short exact sequence
\begin{equation*}
   0\rightarrow R^{q}f_{\ast}(K_{X/Y}\otimes L)_{y}\rightarrow f_{\ast}(\Omega^{n-q}_{X/Y}\otimes L)_{y}\rightarrow\mathcal{Q}^{q}_{y}\rightarrow0.
\end{equation*}
By Corollary 1.7 in \cite{Har80}, $f_{\ast}(\Omega^{n-q}_{X/Y}\otimes L)_{y}$ is reflexive, and hence $\mathcal{Q}^{q}_{y}$ is torsion-free. Therefore $R^{q}f_{\ast}(K_{X/Y}\otimes L)_{y}$ is normal and so reflexive.
\end{proof}

\section{Vanishing Theorem}
As is stated in Theorem \ref{t4}, $R^{q}f_{\ast}(K_{X/Y}\otimes L)$ is reflexive if the singularity of $L$ is tame enough.
In this section, we consider the positivity of the higher direct images. We first prove Theorem \ref{t5} by using Theorems \ref{t3} and \ref{t4}.

There are some issues we should pay attention to before we start the proof. Firstly, in general if $f$ is a surjective morphism, it is generic smooth, namely it is smooth on a Zariski open subset $W\subset X$. Secondly, if $\mathscr{I}(h)=\mathcal{O}_{X}$, it is only assured that $\mathscr{I}(h_{y})=\mathcal{O}_{X_{y}}$ for all $y\in Z:=\{y\in Y|~ h_{y}:=h|_{X_{y}}\textrm{ is well-defined}\}$. Therefore the conditions in Theorem \ref{t4} are not fully satisfied under the assumptions of Theorem \ref{t5}. Fortunately, this is enough for our purposes.

\begin{proof}[Proof of Theorem \ref{t5}]
By asymptotic Serre vanishing theorem, we can choose a positive integer $m_{0}$ such that for all $m\geqslant m_{0}$,
\begin{equation*}
    H^{i}(Y,R^{q}f_{\ast}(K_{X}\otimes L)\otimes\mathcal{O}_{Y}(mA))=0
\end{equation*}
for $i>0,\, q\geqslant0$.
Fix an integer $m$ such that $m\geqslant m_{0}$ and  $\mathcal{O}_{Y}(mA)$ is very ample.

We prove the theorem by induction on $n=\dim X$, the case $n=0$ being trivial.
Denote $A'=f^\ast (A)$ and let $H^{\prime}\in|mA^{\prime}|$ be the pullback of a general divisor $H\in|mA|$.
It follows from Bertini's theorem that we can assume $H$ is integral and $H^{\prime}$ is smooth (though possibly disconnected).
Moreover, we can let  $H$ be disjoint with $Y-Z\cap W$. Then we have a short exact sequence
\begin{equation}\label{e3}
\begin{split}
0&\rightarrow K_{X}\otimes L\otimes\mathcal{O}_{X}(A^{\prime})\rightarrow K_{X}\otimes L\otimes\mathcal{O}_{X}((m+1)A^{\prime})\\
&\rightarrow K_{H^{\prime}}\otimes L\otimes\mathcal{O}_{X}(A^{\prime})|_{H^{\prime}}\rightarrow0
\end{split}
\end{equation}
which is induced by multiplication with a section defining $H^{\prime}$. We get from this short exact sequence a long exact sequence
\begin{equation}\label{01}
\begin{split}
0&\rightarrow f_{\ast}(K_{X}\otimes L\otimes\mathcal{O}_{X}(A^{\prime}))\rightarrow f_{\ast}(K_{X}\otimes L\otimes\mathcal{O}_{X}((m+1)A^{\prime}))\\
&\rightarrow f_{\ast}(K_{H^{\prime}}\otimes L\otimes\mathcal{O}_{X}(A^{\prime})|_{H^{\prime}})\rightarrow R^{1}f_{\ast}(K_{X}\otimes L\otimes\mathcal{O}_{X}(A^{\prime}))\\
&\rightarrow R^{1}f_{\ast}(K_{X}\otimes L\otimes\mathcal{O}_{X}((m+1)A^{\prime}))\rightarrow\cdots
\end{split}
\end{equation}
By Theorem \ref{t4} all the higher direct images of $K_{X}\otimes L\otimes\mathcal{O}_{X}(A^{\prime})$ are torsion-free on $Z\cap W$.
Also clearly the sheaves $R^{q}f_{\ast}(K_{H^{\prime}}\otimes L\otimes\mathcal{O}_{X}(A^{\prime})|_{H^{\prime}})$ are torsion on $H$.
Hence the long exact sequence (\ref{01}) can be split into a family of short exact sequences:  for all $q\geq 0$,
\begin{equation}\label{e4}
\begin{split}
0&\rightarrow R^{q}f_{\ast}(K_{X}\otimes L\otimes\mathcal{O}_{X}(A^{\prime}))\rightarrow R^{q}f_{\ast}(K_{X}\otimes L\otimes\mathcal{O}_{X}((m+1)A^{\prime}))\\
&\rightarrow R^{q}f_{\ast}(K_{H^{\prime}}\otimes L\otimes\mathcal{O}_{X}(A^{\prime})|_{H^{\prime}})\rightarrow0.
\end{split}
\end{equation}
On the other hand, applying the inductive hypothesis to each connected component of $H^{\prime}$,
we have that for all $i\geqslant1$
\begin{equation*}
   H^{i}(Y,R^{q}f_{\ast}(K_{H^{\prime}}\otimes L\otimes\mathcal{O}_{X}(A^{\prime})|_{H^{\prime}}))=0.
\end{equation*}
Furthermore, by the choice of $m$ we also have for all $i\geqslant1$
\begin{equation}\label{02}
    H^{i}(Y,R^{q}f_{\ast}(K_{X}\otimes L\otimes\mathcal{O}_{X}((m+1)A^{\prime})))=0.
\end{equation}
Now by taking the cohomology long exact sequence from the short exact sequence (\ref{e4}), we find for every $i>1$
\begin{equation*}
   H^{i}(Y,R^{q}f_{\ast}(K_{X}\otimes L\otimes\mathcal{O}_{X}(A^{\prime})))=0.
\end{equation*}
This proves the theorem for the cases where $i>1$.

To prove the case where $i=1$, we denote
\begin{equation*}
   B_{l}:=H^{1}(Y,R^{q}f_{\ast}(K_{X}\otimes L\otimes\mathcal{O}_{X}(lA^{\prime}))).
\end{equation*}
By identity (\ref{02}) for $i=1$, we have $B_{m+1}=0$. Hence we consider the following commutative diagram:
\begin{equation*}
\xymatrix{
    B_{1}\ar[d] & \stackrel{\phi}{\longrightarrow} & H^{q+1}(X,K_{X}\otimes L \otimes\mathcal{O}_{X}(A^{\prime}))\ar[d]^{\psi} \\
    B_{m+1} & {\longrightarrow} &H^{q+1}(X,K_{X}\otimes L\otimes\mathcal{O}_{X}((m+1)A^{\prime}))
    }
\end{equation*}
Here the horizontal maps are the canonical injective maps coming out of the Leray spectral sequence, and the vertical maps are induced by multiplication with sections defining $H^{\prime}$ and $H$ respectively.
By Theorem \ref{t3} the map $\psi$ is injective, and hence the composition $\psi\circ\phi$ is also injective.
So $B_1=0$ and we finish the proof of the theorem for the case where $i=1$.
\end{proof}

Using Theorem \ref{t5}, we can prove the positivity of the higher direct images. We first review the definition and a basic result of the Castelnuovo--Mumford regularity \cite{Mum66}.

\begin{definition}
Let $X$ be a projective manifold and $L$  an ample and globally generated line bundle on $X$.
Given an integer $m$, a coherent sheaf $F$ on $X$ is $m$-regular with respect to $L$ if for all $i\geqslant1$
\begin{equation*}
H^{i}(X,F\otimes L^{m-i})=0.
\end{equation*}
\end{definition}

\begin{theorem}(Mumford, \cite{Mum66})\label{t7}
Let $X$ be a projective manifold and $L$ an ample and globally generated line bundle on $X$.
If $F$ is a coherent sheaf on $X$ that is $m$-regular with respect to $L$, then the sheaf $F\otimes L^{m}$ is globally generated.
\end{theorem}

After this, we can  prove Theorem \ref{t6} as a corollary of Theorem \ref{t5}.
\begin{proof}[Proof of Theorem \ref{t6}]
It follows from Theorem \ref{t5} that for every $i\geqslant1$
\begin{equation*}
  H^{i}(Y,R^{q}f_{\ast}(K_{X}\otimes L)\otimes A^{m-i}\otimes A^{\prime})=0.
\end{equation*}
Hence the sheaf $R^{q}f_{\ast}(K_{X}\otimes L)\otimes A^{m}\otimes A^{\prime}$ is $0$-regular with respect to $A$.
So it is globally generated by Theorem \ref{t7}.
\end{proof}

\address{

\small Current address: School of Mathematical Sciences, Fudan University, Shanghai 200092, People¡¯s
Republic of China

\small E-mail address: jingcaowu13@fudan.edu.cn
}

\end{document}